\documentclass[11pt]{amsart}

\usepackage{amsmath,amscd,amssymb,latexsym,color}
\usepackage{longtable}

\newcommand{\sm}{\left(\smallmatrix}
\newcommand{\esm}{\endsmallmatrix\right)}
\newcommand{\px}{\begin{pmatrix}}
\newcommand{\epx}{\end{pmatrix}}

\newcommand{\G}{\Gamma}
\newcommand{\D}{\Delta}

\newcommand{\lt}{\left}
\newcommand{\rt}{\right}

\newcommand{\Z}{\mathbb Z}

\newcommand{\R}{\mathbb R}
\newcommand{\N}{\mathbb N}

\newcommand{\SL}{\text{SL}}

\newcommand{\ord}{\text{ord}}

\newtheorem{thm}{Theorem}
\newtheorem{lem}[thm]{Lemma}
\newtheorem{cor}[thm]{Corollary}
\newtheorem{prop}[thm]{Proposition}

\newtheorem{alg}[thm]{Algorithm}
\theoremstyle{definition}
\newtheorem{df}[thm]{Definition}
\theoremstyle{remark}

\numberwithin{equation}{section}
\numberwithin{thm}{section}

\date{\today}
\address{Bo-Hae Im, Department of Mathematics, Chung-Ang University, 84 Heukseok-Ro, Dongjak-Gu, Seoul 156-756, South Korea}
\email{imbh@cau.ac.kr}

\address{Daeyeol Jeon, Department of
Mathematics education, Kongju National University, 56
Gongjudaehak-ro, Gongju-si, Chungcheongnam-do 314-701, South Korea}
\email{dyjeon@kongju.ac.kr}

\address{Chang Heon Kim,
Department of Mathematics, Sungkyunkwan University, Suwon 440-746, South Korea}
\email{chhkim@skku.edu}

\thanks{Bo-Hae Im was supported by Basic Science Research Program through the National Research Foundation of Korea (NRF) funded
        by the Ministry of Education, Science and Technology (NRF-2014R1A1A2053748).}

\begin{document}

\title{Notes on Weierstrass points of modular curves $X_0(N)$}

\author{Bo-Hae Im, Daeyeol Jeon and Chang Heon Kim}
\maketitle
\begin{abstract}
We give conditions when the fixed points by the partial Atkin-Lehner involutions on $X_0(N)$ are Weierstrass points as an extension of the result by Lehner and Newman \cite{LN}.
Furthermore, we complete their result by determining whether the fixed points by the full Atkin-Lehner involutions on $X_0(N)$ are Weierstrass points or not.
\end{abstract}

\section{Introduction}\label{intro}
Let $\mathfrak H$ be the complex upper half plane and $\Gamma$ be a congruence subgroup of the full modular group $\SL_2(\mathbb Z)$.  Denote by $X(\Gamma)$ the modular curve  obtained by compactifying the quotient
$\Gamma\backslash \mathfrak H$. We can view $X(\Gamma)$ as a compact Riemann surface analytically. For each positive integer $N$, we let $\Gamma_0(N)$ be the Hecke subgroup of $\SL_2(\mathbb{Z})$ defined by
$$
\Gamma_0(N)=\left\{
\px a&b\\c& d\epx \in \SL_2(\mathbb{Z}) \, | \, c\equiv 0 \pmod{N} \right\}
$$
 and let $ X_0(N):=X(\Gamma_0(N))$.

We recall the definitions of a Weierstrass point  of a compact Riemann surface and its Weierstrass weight as in \cite{AO}.  One says that a point $P$ of a compact Riemann surface $X$ of genus $g\ge 2$ is  \textit{a Weierstrass point} of $X$ if there is a holomorphic differential $\omega$ (not identically zero) with a zero of order $\geq g$ at $P$. If $P\in X$ and $\omega_1,\omega_2,\ldots, \omega_g$ form a basis for the holomorphic differentials  on $X$ with the property that
$$0=\ord_P(\omega_1)<\ord_P(\omega_2)<\cdots <\ord_P(\omega_g),$$
then \textit{the Weierstrass weight of $P$} is the non-negative integer defined by
$$\text{wt}(P):=\sum\limits_{j=1}^g(\ord_P(\omega_j)-j+1).$$
Note that $\text{wt}(P)>0$ if and only if $P$ is a Weierstrass point of $X$.

The Weierstrass points of modular curves have been studied:  Lehner and Newman \cite{LN} has given conditions when the cusp at infinity is a Weierstrass point of $X_0(N)$ for $N=4n, 9n$ and Atkin \cite{A} has given  conditions for the case of $N=p^2n$ where $p$ is a prime $\geq 5$. Also, Ogg \cite{Ogg}, Kohnen \cite{Kohnen1, Kohnen2} and Kilger~\cite{Kilger} have given some conditions when the cusp at infinity is not a Weierstrass point of $X_0(N)$ for certain $N$. And  Ono \cite{Ono} and Rohrlich \cite{Rohrlich} have investigated Weierstrass points of $X_0(p)$ for various primes $p$. Recently  in \cite{IJK} we investigated Weierstrass points on coverings of $X_0(N)$.

The Weierstrass points have been illustrated as an important class in number theory. For example, they have been used to determine the group structure of the automorphism groups of compact Riemann surfaces.

The main purpose of this paper is to give some conditions of being the Weierstrass points of $X_0(N)$.  To describe our main result in this paper more precisely, we recall the definition of the Atkin-Lehner involution.  We call a positive divisor $Q$ of $N$ with $\gcd(Q,N/Q)=1$ an {\it exact divisor} of $N$ and denote it by $Q\| N$.
For each $Q\|N$, consider the matrices of the form $\begin{pmatrix} Qx & y\\Nz & Qw \end{pmatrix}$ with $x,y,z,w\in\mathbb Z$ and determinant $Q.$ Then such a matrix defines a unique involution on $X_0(N)$ which is called the {\it Atkin-Lehner involution} and denoted by $W_Q.$
If $Q=N,$ then $W_Q$ is called the {\it full Atkin-Lehner involution,} and otherwise it is said to be a {\it partial Atkin-Lehner involution}.
Sometimes we regard $W_Q$ as a matrix.

Lehner and Newman \cite{LN} have shown that the fixed points by the full Atkin-Lehner involution on $X_0(N)$ are Weierstrass points except possibly for finitely many $N$ which are listed in Lemma~\ref{full}. In order to obtain the result, they have applied Sch\"{o}neberg's Theorem (\cite{SCH}). Our main result in this paper is to determine explicitly when the fixed points by the partial Atkin-Lehner involutions on $X_0(N)$ are Weierstrass points as an extension of the result by Lehner and Newman \cite{LN}. Also in this paper we give an algorithm to generate $\Gamma_0(N)$-inequivalent fixed points by the full Atkin-Lehner involution and we provide a computational method which can take care of the exceptional cases listed in Lemma~\ref{full} which are not covered in the Lehner and Newman \cite{LN}.

This paper is organized as follows. In Section~\ref{fixed}, we recall the formula (\cite{FH}) for the number of fixed points on $X_0(N)$ by the Atkin-Lehner involutions and explain algorithmically how to generate $\Gamma_0(N)$-inequivalent fixed points of $X_0(N)$ by the full Atkin-Lehner involution $W_N$.  In Section~\ref{weier},  we give conditions when fixed points by the Atkin-Lehner involutions are Weierstrass points. We apply the formula (\cite{FH}) of the number of fixed points on $X_0(N)$ by the Atkin-Lehner involutions and Sch\"{o}neberg's Theorem (\cite{SCH}), and we give some new formulae for the number of fixed points when so-called {\it the elliptic condition} (Definition~\ref{elliptic}) is satisfied and apply them to obtain the conditions for Weierstrass points. In Section~\ref{comp}, we consider the exceptional cases  which are not determined solely by Sch\"{o}neberg's Theorem (\cite{SCH}) and the formula given in \cite{FH}, and we give a computational explanation how to determine whether they are Weierstrass points or not.

\section{Fixed points by Atkin-Lehner involutions}\label{fixed-points}\label{fixed}

Let $X_0^{+Q}(N)$ be the quotient space of $X_0(N)$ by $W_Q$. Let $g_0(N)$ and $g_0^{+Q}(N)$ be the genera of $X_0(N)$ and $X_0^{+Q}(N)$ respectively.
Then $g_0^{+Q}(N)$ is computed by the Riemann-Hurwitz formula as follows:
$$g_0^{+Q}(N)=\frac{1}{4}(2g_0(N)+2-\nu(Q)),$$
where $\nu(Q):=\nu(Q;N)$ is the number of fixed points on $X_0(N)$ by $W_Q$. We recall the formula for $\nu(Q)$.

\begin{prop}\label{FH}$($\cite{FH}$)$
For each $Q\|N$, $\nu(Q)$ is given by
\begin{align}\label{genus} \notag
\nu(Q)=&\left(\prod_{p|N/Q}c_1(p)\right)h(-4Q) \\ \notag
&+\begin{cases}\left(\prod_{p|N/Q}c_2(p)\right)h(-Q),&\text{if}\,\, 4\leq Q\equiv3\pmod{4},\\
0,&\text{otherwise}\end{cases}\\
&+\begin{cases}\prod_{p|N/2}\left(1+\left(\frac{-4}{p}\right)\right),&\text{if}\,\, Q=2,\\
0,&\text{otherwise}\end{cases}\\ \notag
&+\begin{cases}\prod_{p|N/3}\left(1+\left(\frac{-3}{p}\right)\right),&\text{if}\,\, Q=3,\\
0,&\text{otherwise}\end{cases}\\ \notag
&+\begin{cases}\prod_{p^k|N/Q}\left(p^{[\frac{k}{2}]}+p^{[\frac{k-1}{2}]}\right),&\text{if}\,\, Q=4,\\
0,&\text{otherwise}\end{cases}
\end{align}
where $h(-d)$ is the class number of primitive quadratic forms of discriminant $-d$, $(\frac{\cdot}{\cdot})$ is the Kronecker symbol and the functions $c_i(p)$ are defined as follows: for $i=1,2$,
\begin{align*}
c_i(p)=&\begin{cases}1+\left(\frac{-Q}{p}\right),&\text{if}\,\, p\neq 2\,\,\text{and}\,\, Q\equiv3\pmod{4},\\
1+\left(\frac{-4Q}{p}\right),&\text{if}\,\, p\neq 2\,\,\text{and}\,\, Q\not\equiv3\pmod{4},\end{cases}\\
c_1(2)=&\begin{cases}1,&\text{if}\,\, Q\equiv1\pmod{4}\,\,\text{and}\,\, 2\|N,\\
0,&\text{if}\,\, Q\equiv1\pmod{4}\,\,\text{and}\,\, 4|N,\\
2,&\text{if}\,\, Q\equiv3\pmod{4}\,\,\text{and}\,\, 2\|N,\\
3+\left(\frac{-Q}{2}\right),&\text{if}\,\, Q\equiv3\pmod{4}\,\,\text{and}\,\, 4\|N,\\
3\left(1+\left(\frac{-Q}{2}\right)\right),&\text{if}\,\, Q\equiv3\pmod{4}\,\,\text{and}\,\, 8|N,
\end{cases}\\
c_2(2)=&1+\left(\frac{-Q}{2}\right),\ \ \text{if}\,\, Q\equiv3\pmod{4}.
\end{align*}

\end{prop}

\

Next we give an algorithm to find $\Gamma_0(N)$-inequivalent fixed points of $W_N$ on $X_0(N)$.

For a positive integer $d$ conguent to $0$ or $3$ modulo $4$, we denote by $\mathcal Q_d$
the set of positive definite integral binary quadratic forms
$$Q(x,y)=[a,b,c]:=ax^2+bxy+cy^2$$
with discriminant $-d=b^2-4ac$.
Then ${\rm SL}_2(\mathbb Z)$ acts on $\mathcal Q_d$ by
$$Q\circ\gamma(x,y)=Q(px+qy,rx+sy)$$
where $\gamma=\begin{pmatrix}p&q\\r&s\end{pmatrix}\in {\rm SL}_2(\mathbb Z)$.
We say that a quadratic form $[a,b,c]\in \mathcal Q_d$ is {\it primitive} if $\gcd(a,b,c)=1$.
A primitive
positive definite form $[a,b,c]$ is said to be {\it reduced}
if
\begin{equation} \label{RED}
|b|\leq a\leq c,\,\,\mbox{and}\,\,b\geq 0\,\,\mbox{if either}\,\,|b|=a\,\,\mbox{or}\,\,a=c.
\end{equation}
Let $\mathcal Q_d^\circ\subset\mathcal Q_d$ be the subset of primitive forms.
Then ${\rm SL}_2(\mathbb Z)$ also acts on $\mathcal Q_d^\circ$.
As is well known, there is a 1-1 correspondence between the set of classes $\mathcal Q_d^\circ/{\rm SL}_2(\mathbb Z)$ and the set of reduced forms.

Now for a fixed positive integer $N$ we let $d$ be a positive integer such that $-d$ is congruent to a square modulo $4N$. We choose an integer $\beta$ with $-d\equiv \beta^2 \pmod{4N}$. Then we define
$$
\mathcal Q_{d,N}^\circ=\{[aN,b,c]\in\mathcal Q_d\,|\,\gcd(a,b,c)=1\}
$$
and
$$\mathcal Q_{d,N,\beta}^\circ=\{[aN,b,c]\in\mathcal Q_{d,N}^\circ \,|\,\     b\equiv \beta \pmod{2N} \}.$$
Then $\Gamma_0(N)$ acts on both $\mathcal Q_{d,N}^\circ$ and $\mathcal Q_{d,N,\beta}^\circ$.
For $Q=[aN, b, c]\in \mathcal Q_{d,N}^\circ$ we define $Q|{W_N}=[cN, -b, a]$. We observe that this defines an action of $W_N$ on the set $\mathcal Q_{d,N}^\circ/\Gamma_0(N)$.
\par
Assume that $N\ge 5$ and $W_N$ fixes $\Gamma_0(N)\tau \in X_0(N)$ for some $\tau\in\mathfrak H$. This means that
$
\sm p&q \\ r&s \esm W_N \tau= \tau
$
for some $\sm p&q \\ r&s \esm \in \Gamma_0(N)$.
Thus $\tau$ satisfies a quadratic equation $$Ns X^2 -(r+qN) X +p =0,$$ whose discriminant $D$ is given by $D=(r+qN)^2-4psN=(r'-q)^2 N^2 -4N$ where $r=r' N$.
Since $D\le 0$,  we come up with $(r'-q)^2 N \le 4$. Since we have assumed that $N\ge 5$, we must have  $r'-q=0$ and therefore $D=-4N$. Thus one has $D=4q^2 N^2 -4psN=-4N$, which gives $\gcd(p,qN)=1$ and hence $\gcd(s,-2qN,p)=1$ or $2$. If $\gcd(s,-2qN, p)=2$, then both $p$ and $s$ should be even. It then follows from $ps-q^2 N=1$ that $N\equiv 3 \pmod{4}$. Now we set $Q_\tau:=[sN, -2qN, p]$. Then we see that
$$
 Q_\tau \in \mathcal Q_{4N,N,0}^\circ
\, \hbox{ or } \, \frac 12 Q_\tau \in \mathcal Q_{N,N,N}^\circ
$$
where the latter case may happen only if $N\equiv 3 \pmod{4}$.
Let $\mathcal F_N$ be the set of the fixed points of $W_N$ on $X_0(N)$ and let
\begin{equation*}
\mathcal G_N:=\begin{cases} \mathcal Q_{4N,N,0}^\circ/\Gamma_0(N)\cup \mathcal Q_{N,N,N}^\circ/\Gamma_0(N),&\text{if } N\equiv 3\pmod 4\\
\mathcal Q_{4N,N,0}^\circ/\Gamma_0(N),&\text{otherwise},
\end{cases}
\end{equation*}
where the union is disjoint.
Now we define a map $\phi:\mathcal F_N\to \mathcal G_N$ by
\begin{equation*}
\phi(\Gamma_0(N)\tau)=\begin{cases} Q_\tau,&\text{if } (s,-2qN,p)=1\\
\frac{1}{2}Q_\tau,&\text{if } (s,-2qN,p)=2.
\end{cases}
\end{equation*}
Then we can check easily $\phi$ is well-defined.

On the other hand, given a quadratic form $Q\in
\mathcal Q_{4N,N,0}^\circ/\Gamma_0(N)$ or $Q\in \mathcal Q_{N,N,N}^\circ/\Gamma_0(N)$ we can find a fixed point $\Gamma_0(N)\tau \in X_0(N)$ such that $Q=\phi(\Gamma_0(N)\tau)$.
Thus the map $\phi$ is surjective.

Now we deduce from \cite[Proposition in p.505]{GKZ} that the natural maps
\begin{equation}\label{map}
\begin{array}{l}
\mathcal Q_{4N,N,0}^\circ/\Gamma_0(N) \to \mathcal Q_{4N}^\circ /{\rm SL}_2(\mathbb Z) \text{ and }\\
\mathcal Q_{N,N,N}^\circ/\Gamma_0(N) \to \mathcal Q_{N}^\circ /{\rm SL}_2(\mathbb Z) \text{ (when $N\equiv 3 \pmod{4}$)}
\end{array}
\end{equation}
are well-defined and bijective.

From \cite{FH}, we have the following formula for the number of fixed points of $W_N$ on $X_0(N)$:
\begin{equation*}
\nu(N)=\#\mathcal F_N=\delta_N h(-4N),
\end{equation*}
where $h(-4N)$ denotes $\#\mathcal Q_{4N}^\circ /{\rm SL}_2(\mathbb Z)$ and
\begin{equation*}
\delta_N=\begin{cases}2,&\text{if}\,\,N\equiv7\pmod{8}\\
\frac{4}{3},&\text{if}\,\,N\equiv3\pmod{8} \text{ and } N>3,\\
1,&\text{otherwise}.
\end{cases}
\end{equation*}
Since it is well-known that
\begin{equation}\label{order}
h(-4N)=\begin{cases}h(-N),&\text{if}\,\,N\equiv7\pmod 8\\
3h(-N),&\text{if}\,\, N\equiv3\pmod{8} \text{ and } N>3,\end{cases}
\end{equation}
we have $\#\mathcal F_N=\#\mathcal G_N$, and hence $\phi$ is a bijection.

Now we will find $\Gamma_0(N)$-inequivalent fixed points in $\mathcal F_N$ by finding $\Gamma_0(N)$-inequivalent quadratic forms in $\mathcal G_N$ which can be obtained by pulling back the reduced forms in $\mathcal Q_{4N}^\circ /{\rm SL}_2(\mathbb Z)$ and $\mathcal Q_{N}^\circ /{\rm SL}_2(\mathbb Z)$ through the maps \eqref{map}.
Before providing an algorithm, we need the following lemma.
\begin{lem} \label{alg1_lem}
 For a fixed positive integer $N$ we let $d$ be a positive integer such that $-d$ is congruent to a square modulo $4N$. We choose an integer $\beta$ with $-d\equiv \beta^2 \pmod{4N}$. Then
the following statements are true.
 \begin{enumerate}
 \item[(1)] Given a primitive quadratic form $Q\in \mathcal Q_d^\circ$ there exists a quadratic form $[a,b,c]$ which is ${\rm SL}_2(\mathbb Z)$-equivalent with $Q$ and $\gcd(a, N)=1$.
\item[(2)] Let $K$ be a solution to the linear congruence equation $2aX+b\equiv -\beta \pmod{2N}$ and set $[A,B,C]:=[a,b,c]\circ \sm K&-1\\1&0 \esm$. Then $[A,B,C]$ belongs to
$\mathcal{Q}_{d,N,\beta}.$
\end{enumerate}
\end{lem}
\begin{proof} (1) follows from \cite[Lemma 2.3 and Lemma 2.25]{C}.
\par\noindent
(2) First we note that $b$ and $\beta$ have the same parity since $-d=b^2-4ac$ and $-d\equiv \beta^2 \pmod{4N}$. Thus we obtain that $2=\gcd(2a, 2N) \mid (-\beta-b)$, which guarantees that the linear congruence $2aX+b\equiv - \beta \pmod{2N}$ is solvable. Since $B=-2aK-b\equiv \beta \pmod{2N}$ and $C=a$, we must have
$$
-4Aa=-4AC=-d-B^2\equiv \beta^2-B^2\equiv 0 \pmod{4N},
$$
which yields that $N|A$ since $\gcd(a,N)=1$. Hence one has $[A,B,C]\in \mathcal{Q}_{d,N,\beta}$, as desired.
\end{proof}

Now we summarize the procedures explained in the above as the following algorithm.

\begin{alg}\label{alg1}
The following steps implement an algorithm to find $\Gamma_0(N)$-inequivalent fixed points of $W_N$:
\par\noindent
{\bf STEP 1.} Set $(d,\beta)=(4N, 0)$ or $(d,\beta)=(N,N)$ (when $N\equiv 3 \pmod{4}$).
\par\noindent
{\bf STEP 2.} Starting from a reduced form $Q^{{\rm red}}$ satisfying \eqref{RED} we first find a quadratic form $[a,b,c]$ which is ${\rm SL}_2(\mathbb Z)$-equivalent with $Q^{{\rm red}}$ and $\gcd(a,N)=1$.
\par\noindent
{\bf STEP 3.} Set $[A,B,C]:=[a,b,c]\circ \sm K&-1\\1&0 \esm$ where $K$ is a solution to the linear congruence equation$2aX+b\equiv -\beta \pmod{2N}$. Then $[A,B,C]$ belongs to
$\mathcal{Q}_{d,N,\beta}.$
\par\noindent
{\bf STEP 4.} Let $\tau=\frac{-B+\sqrt{-d}}{2A}$. Then $\Gamma_0(N)\tau$ gives a fixed point of $W_N$.
\end{alg}

\section{Weierstrass point on $X_0(N)$ arising from the fixed points of Atkin-Lehner involutions}\label{weier}

First, we recall Sch\"{o}neberg's Theorem \cite{SCH} as follows.

\begin{thm}\label{sch}{$($\cite{SCH}$)$} Let $g$ be the genus of  the modular curve $X(\Gamma)$ of a congruence subgroup $\Gamma$ of $\mathrm{SL}_2(\Z)$ and let $M$ be an element of the normalizer of $\Gamma$ in $\mathrm{SL}_2(\R)$ and $p$ be the exponent of $M$ modulo $\Gamma$. Let $g^*$ be the genus of the quotient space $X(\Gamma)/\langle M\rangle$ by the subgroup $\langle M \rangle$ generated by $M$ modulo $\Gamma$. Then $\tau$ is a Weierstrass point on  $X(\Gamma)$ provided that
$$g^*\neq \lfloor g/p\rfloor.$$
$($$\lfloor x\rfloor$ denotes the largest integer not greater than $x$.$)$
\end{thm}

By using the formula in Proposition~\ref{FH} and Sch\"{o}neberg's Theorem, we have the following:

\begin{lem}\label{fixed points}
Let $\tau$ be a fixed point of $W_Q$ on $X_0(N)$ with $g_0(N)>1$.
If $\nu(Q)>4$, then $\tau$ is a Weierstrass point of $X_0(N)$.
\end{lem}
\begin{proof}
By Theorem~\ref{sch}, $\tau$ is a Weierstrass point of $X_0(N)$ if
$$g_0(N)-2g_0^{+Q}(N)>1$$
which is equivalent to that $\nu(Q)>4$.
\end{proof}

Note that $\nu(N)=\delta_Nh(-4N)$ and there exist only finitely many $N$ such that $h(-4N)\leq 4$.
By using these facts, Lehner and Newman \cite{LN} have shown that the fixed points of $W_N$ are Weierstrass points on $X_0(N)$ except possibly for finitely many $N$.
However they didn't specify such possible $N$'s, and hence we list them in Lemma \ref{full}.
By  Proposition \ref{number} which we will prove later, we have the following:

\begin{lem}\label{full} The fixed points of $W_N$ are Weierstrass points on $X_0(N)$ with $g_0(N)>1$ except possibly for the following values for $N$:
\begin{quote}
$22$, $28$, $30$, $33$, $34$, $37$, $40$, $42$, $43$, $45$, $46$, $48$, $52$, $57$, $58$, $60$, $64$, $67$, $70$, $72$, $73$, $78$, $82$, $85$, $88$, $93$, $97$, $100$, $102$, $112$, $130$, $133$, $142$, $148$, $163$, $177$, $190$, $193$, $232$, $253$.
\end{quote}
\end{lem}

From now on, if $Q\|N$ and $N=QM$, we always assume $M>1$.
By  Proposition \ref{genus} and Lemma \ref{fixed points} we have the following results.

\begin{lem}\label{AL-2} Let $2\|N$ and $N=2M$. Then we have the following.
\begin{enumerate}
 \item[(1)] $\nu(2)=0$ if and only if $M$ has a prime factor $p$ with $p\equiv 7\pmod 8$ or prime factors $q,r$ with $q\equiv 3\pmod 8$ and $r\equiv 5\pmod 8$.
\item[(2)] If $\nu(2)\neq 0$, then
$$\nu(2)=\delta_{0,s_2}2^{s_0+s_1}+\delta_{0,s_1}2^{s_0+s_2},$$
where $s_0$, $s_1$ and $s_2$ are the numbers of prime factors $p$ of $M$ with $p\equiv1\pmod 8$, $p\equiv3\pmod 8$ and $p\equiv5\pmod 8$ respectively, and $\delta_{i,j}$ is the Kronecker  delta function.
\end{enumerate}
\end{lem}
\begin{proof} Since $h(-8)=1$, $$\nu(2)=\prod_{p|M}\left(1+\left(\frac{-8}{p}\right)\right)+\prod_{p|M}\left(1+\left(\frac{-4}{p}\right)\right).$$
Since $$\left(\frac{-8}{p}\right)=\begin{cases}1,& p\equiv1,3\pmod{8}\\
-1,&p\equiv5,7\pmod{8}\end{cases}$$
and
$$\left(\frac{-4}{p}\right)=\begin{cases}1,& p\equiv1\pmod{4}\\
-1,&p\equiv3\pmod{4}\end{cases},$$
our result follows.
\end{proof}

\begin{lem}\label{AL-3} Let $3\|N$ and $N=3M$. Then we have the following.
\begin{enumerate}
\item[(1)] $\nu(3)=0$ if and only if $8|M$ or $M$ has a prime factor $p$ with $p\equiv 5$ or $11\pmod {12}$.
\item[(2)] If $\nu(3)\neq 0$, then
$$\nu(3)=2^{s+1},$$
where $s$ is the number of prime factors $p$ of $M$ with $p\equiv 1$ or $7\pmod {12}$.
\end{enumerate}
\end{lem}
\begin{proof} If $8|M$, then $\nu(3)=0$ because $\displaystyle\left(\frac{-3}{2}\right)=-1.$
Suppose $8\nmid M$.
Since $h(-12)=1$,
\begin{equation*}
\nu(3)=\prod_{p|M}\left(1+(-1)^{p-1}\left(\frac{-3}{p}\right)\right)+\prod_{p|M}\left(1+\left(\frac{-3}{p}\right)\right)\\
\end{equation*}
Since $$\left(\frac{-3}{p}\right)=\begin{cases}1,& p\equiv1,7\pmod{12}\\
-1,&p\equiv5,11\pmod{12}\end{cases}$$
for an odd prime $p$, the condition of $\nu(3)>0$ follows.
If $\nu(3)>0$, then
\begin{equation*}
\nu(3)=2\prod_{p|M}\left(1+\left(\frac{-3}{p}\right)\right)=2^{s+1}.\\
\end{equation*}
\end{proof}

\begin{lem}\label{AL-4} Let $4\|N$ and $N=4M$.
Then $W_4$ has always a fixed point on $X_0(N)$.
Let $s$ and $t$ denote the numbers of prime factors $p$ of $M$ with $p\equiv 1\pmod {4}$ and $p\equiv3\pmod {4}$ respectively.
Then we have the following:
$$\nu(4)=\begin{cases} \prod_{p^k||M}\left(p^{\left[\frac{k}{2}\right]}+p^{\left[\frac{k-1}{2}\right]}\right), &\text{ if } t>0\\
\prod_{p^k||M}\left(p^{\left[\frac{k}{2}\right]}+p^{\left[\frac{k-1}{2}\right]}\right)+2^s, &\text{ if } t=0.
\end{cases}$$
\end{lem}
\begin{proof}
It follows directly from Proposition \ref{genus}.
\end{proof}

By using lemmas \ref{AL-2}, \ref{AL-3}, and \ref{AL-4} we have the following result:

\begin{thm} Let $Q\|N$ and $N=QM$. Suppose $\nu(Q)>0$.
Then the fixed points of $W_Q$ are Weierstrass points of $X_0(N)$ for each of the followings:
\begin{enumerate}
\item[(1)] $Q=2;$ $s_0>1$ and $s_1=s_2=0$, $s_0+s_1>2$ or $s_0+s_2>2$ where $s_0$, $s_1$ and $s_2$ are the numbers of prime factors $p,q$ and $r$ of $M$ with $p\equiv1\pmod 8$, $q\equiv3\pmod 8$ and $r\equiv5\pmod 8$ respectively.
\item[(2)] $Q=3;$ $s>1$ where $s$ is the number of prime factors $p$ of $M$ with $p\equiv 1,7\pmod {12}$.
\item[(3)] $Q=4;$ $M$ is not a square-free integer with $M\neq 9$ or $M$ is square-free and $6s+4t>11$ where $s$ and $t$ are the numbers of prime factors $p$ and $q$ of $M$ with $p\equiv 1\pmod {4}$ and $q\equiv3\pmod {4}$ respectively.
\end{enumerate}
\end{thm}

\begin{proof} (1) and (2) follow immediately from Lemma \ref{AL-2} and Lemma \ref{AL-3}.


For (3), if $M$ is not a square-free integer with $M\neq 9$, then one can check easily that $\nu(4)>4$ by Lemma \ref{AL-4}.
If $M$ is square-free, then by Lemma \ref{AL-4}
$$\nu(4)=\begin{cases} 2^{s+t}, &\text{ if } t>0\\
2^{s+1}, &\text{ if } t=0.
\end{cases}$$
One can check easily that $\nu(4)>4$ if and only if $6s+4t>11$.
\end{proof}

Now we deal with the case when $Q>4$. For that we need the following result.

\begin{prop}\label{Del} Let $Q\|N$ and $N=QM$. If $Q>3$, then the following statements are equivalent.
\begin{enumerate}
\item[(1)] $W_Q$ has a fixed point on $X_0(N)$.
\item[(2)] $W_Q$ can be defined by a matrix of the form $\begin{pmatrix}Qx&y\\Nz&-Qx\end{pmatrix}$ with $\det(W_Q)=Q$.
\item[(3)] $X^2\equiv -Q\pmod N$ has a solution.
\item[(4)] $\displaystyle\left(\frac{-Q}{p}\right)=1$ for any odd prime $p|M$, and $Q\equiv 3\pmod 4$ if $4\|M$ and $Q\equiv 7\pmod 8$ if $8|M$.
\end{enumerate}
\end{prop}
\begin{proof} (1)$\Rightarrow$(2) : Suppose $W_Q=\begin{pmatrix}Qx&y\\Nz&Qw \end{pmatrix}$ with $\det(W_Q)=Q$ and $W_Q$ has a fixed point on $X_0(N)$.
Then there exists $\tau\in\mathfrak H$ such that $W_Q\tau=\gamma\tau$ for some $\gamma\in\Gamma_0(N)$. Since $\gamma^{-1}W_Q$ defines the same partial Atkin-Lehner involution, one may assume that $W_Q\tau=\tau$ by changing coefficients $x,y,z,w$ of $W_Q$.
Then the matrix $W_Q$ is elliptic, and hence $|x+w|\sqrt{Q}<2$.
Since $Q>3$, $x+w=0$ and the result follows.

(2)$\Rightarrow$(3) : Suppose $W_Q=\begin{pmatrix}Qx&y\\Nz&-Qx \end{pmatrix}$ with $\det(W_Q)=Q$. Then $\det(W_Q)=-(Qx)^2-Nyz=Q$, and hence $Qx$ is a solution of $X^2\equiv -Q\pmod N$.

(3)$\Rightarrow$(2) : Suppose $X^2\equiv -Q\pmod N$ has a solution, say $x_0$.
Since $\gcd(Q,M)=1$, one can choose $Q'$ so that $QQ'\equiv1\pmod M$.
Then $QQ'x_0$ is a solution of $X^2\equiv -Q\pmod N$ which is divisible by $Q$. Letting $x=Q'x_0$, $(Qx)^2=-Q-Nyz$ for some $y,z$, and hence the matrix $\begin{pmatrix}Qx&y\\Nz&-Qx \end{pmatrix}$ has $\det(W_Q)=Q$.

(2)$\Rightarrow$(1) : Since $W_Q=\begin{pmatrix}Qx&y\\Nz&-Qx\end{pmatrix}$ with $\det(W_Q)=Q$ is elliptic, it has a fixed point in $\mathfrak H$, and hence has one on $X_0(N)$.

(3)$\Leftrightarrow$(4) : It follows from \cite[Ch.9, Theorem 9.13]{Bu}.

\end{proof}




\begin{df}\label{elliptic} For $Q\|N$, if one of the equivalent statements of Proposition \ref{Del} is satisfied, then we call that $(N;Q)$ satisfies the {\it elliptic condition}.
\end{df}


From Proposition \ref{genus} and Proposition \ref{Del}, we have the following:

\begin{lem}\label{fn} Let $Q\| N$ and $N=QM$. Suppose $Q>4$ and $(N;Q)$ satisfies the elliptic condition. Then the following holds:
Let  $s$ be the number of prime divisors of $M$ and
$$\alpha_N=\begin{cases} 1,& \text{if}\,\, 2\nmid N\,\,\text{or}\,\,2\| N,\\
2,& \text{if}\,\, 4\|N,\\
3,& \text{if}\,\, 8|N. \end{cases}$$
\begin{enumerate}
\item If $Q\equiv 7\pmod{8}\,\,\text{or}\,\,Q\equiv 3\pmod{8}\,\, \text{and}\,\, N \text{ is odd}$, then
$$\nu(Q)=2^s(\alpha_N h(-4Q)+h(-Q)).$$
\item If $Q\equiv 1\pmod{4}\,\, \text{and}\,\, N \text{ is even}$, then
$$\nu(Q)=2^{s-1}h(-4Q).$$
\item If $Q$ is even, or $Q\equiv 3\pmod 8$ and $N$ is even, or $Q\equiv 1\pmod 4$ and $N$ is odd, then
$$\nu(Q)=2^sh(-4Q).$$
\end{enumerate}
\end{lem}
\begin{proof} If $Q$ is even, then $M$ is odd, and hence $c_1(p)=2$ for all $p|M$. Thus $\nu(Q)=2^sh(-4Q)$.
This equality holds for the cases of $Q\equiv 1\pmod 4$ and $N$ is odd because $M$ is odd.
If $Q\equiv 1~\pmod 4$ and $N$ is even, then $M$ is even and $c_1(2)=1$.
Thus $\nu(Q)=2^{s-1}h(-4Q)$.
Consider the case $Q\equiv 3\pmod 4$.
If $N$ is odd, then so is $M$, and $c_1(p)=c_2(p)=2$ for all $p|M$. Thus $\nu(Q)=2^s(h(-4Q)+h(-Q))$.
If $Q\equiv 7\pmod 8$ and $N$ is even, then $M$ is even, and $c_1(2)=2\alpha_N$, $c_2(2)=c_i(p)=2$ for all odd prime $p|M$ and $i=1,2$.
Thus $\nu(Q)=2^s(\alpha_N h(-4Q)+h(-Q))$.
If $Q\equiv 3\pmod 8$ and $N$ is even, then $M$ is even, and $c_2(2)=0$, $c_1(p)=2$ for all $p|M$. Thus $\nu(Q)=2^sh(-4Q)$.
\end{proof}
For getting the condition that $\nu(Q)>4$, we need to determine the values for $Q$ with small $h(-Q)$ and $h(-4Q)$.
Recall that if $d_K$ is the (fundamental) discriminant of an imaginary quadratic field $K$, then $h(d_K)$ is equal to the class number of $K$, i.e. $h(d_K)=h(\mathcal O_K)$ where $\mathcal O_K$ is the ring of integers of $K$ and $h(\mathcal O_K)$ is the order of the ideal class group of $\mathcal O_K$.
On the other hand, if $d=f^2d_K$ is the discriminant of a primitive quadratic form with $f>1$, then $h(d)=h(\mathcal O)$ where $\mathcal O$ is the order of conductor $f$ in $K$ (cf. \cite{C}).
Note that $d$ is a fundamental discriminant if and only if one of the following statements holds:
\begin{enumerate}
\item[$\bullet$] $d\equiv 1\pmod 4$ and $d$ is square-free,
\item[$\bullet$] $d=4m$ where $m\equiv 2$ or $3\pmod 4$ and $m$ is square-free.
\end{enumerate}

The complete list of fundamental discriminants of class number 1 is as follows:
\begin{equation}\label{class1-1}
-3,-4,-7,-8,-11,-19,-43,-67,-163.
\end{equation}
This is accomplished independently by Heegner~\cite{H}, Baker~\cite{B1} and Stark~\cite{S1}.
The non-fundamental discriminants of class number $1$ are as follows:
\begin{equation}\label{class1-2}
-12,-16,-27,-28.
\end{equation}
Thus $h(-Q)=1$ if and only if $Q\in S_1$, where
$$S_1=\{3,4,7,8,11,12,16,19,27,28,43,67,163\}.$$

The determination of fundamental discriminants of class number 2 was done again by Baker\cite{B2} and Stark\cite{S2}.

Now consider the condition that $\nu(Q)>4$.
Due to \eqref{order} and Lemma \ref{fn}, it suffices to determine the values for $Q$ such that $h(-Q)=1$ and $h(-4Q)=2,3,4$.
For the purpose, we refer to a paper by Klaise~\cite{K} in which all the orders of class number $2$ and $3$ are determined and an algorithm to find all orders of class number up to $100$ is suggested. Let
\begin{align*}
S_2=\{&5,6,8,9,10,12,13,16,18,22,25,28,37,58\},\\
S_3=\{&11,19,23,27,31,43,67,163\},\\
S_4=\{&14,17,20,21,24,30,32,33,34,36,39,40,42,45,46,48,49,\\&52,55,57,60,63,64,70,72,73,78,82,85,88,93,97,100,102,
\\&112,130,133,142,148,177,190,193,232,253\}.
\end{align*}

By using the result in \cite{K} and some computations of the orders of class number $4$ we have the following:

\begin{prop}\label{number} The list of $Q$ with $h(-4Q)=2,3,4$ is completely determined as follows:
$$h(-4Q)=i\,\, \text{if and only if}\,\, Q\in S_i\,\, \text{for}\,\,i=2,3,4.$$
\end{prop}

By Proposition \ref{number}, we have the following result:

\begin{thm} Let $Q\| N$ and $N=QM$. Suppose that $Q>4$ and $(N,Q)$ satisfies the elliptic condition. Let $s$ denote the number of prime divisors of $M$.
\begin{enumerate}
\item[(1)] When $Q\equiv7\pmod 8$ or $Q\equiv 3\pmod 8$ and $N$ is odd,
 \begin{enumerate}
 \item[a)] if $Q\neq 7$, then all the fixed points of $W_Q$ on $X_0(N)$ are Weierstrass points, and
 \item[b)] if $Q=7$ and $4|N$ or $s>1$, then all the fixed points of $W_7$ on $X_0(N)$ are Weierstrass points.
\end{enumerate}
\item[(2)] When $Q\equiv1\pmod 4$ and $N$ is even,
\begin{enumerate}
\item[a)] if $Q\notin S_2\cup S_4$, then all the fixed points of $W_Q$ on $X_0(N)$ are Weierstrass points, and
\item[b)] if $Q\in S_2$ and $s>2$ or $Q\in S_4$ and $s>1$, then all the fixed points of $W_Q$ on $X_0(N)$ are Weierstrass points.
\end{enumerate}
\item[(3)] In the other cases,
\begin{enumerate}
\item[a)] if $Q\notin S_2$, then all the fixed points of $W_Q$ on $X_0(N)$ are Weierstrass points, and
\item[b)] if $Q\in S_2$ and $s>1$, then all the fixed points of $W_Q$ on $X_0(N)$ are Weierstrass points.
\end{enumerate}
\end{enumerate}
\end{thm}
\begin{proof} (1) From Lemma \ref{fn}, we need to consider $Q$ such that $h(-4Q)+h(-Q)=2$.
By Proposition \ref{number}, one can check that the only $Q=7$ such that $h(-4Q)+h(-Q)=2$. In the case of $Q=7$, if $4|N$ or $s>1$, then $\nu(7)>4$, and hence the result follows.

(2) From Lemma \ref{fn}, we need to consider $Q$ such that $h(-4Q)\leq 4$.
By Proposition \ref{number}, one can check that there don't exist $Q\equiv 1\pmod 4$ such that $h(-4Q)=1,3$, and hence the result follows.

(3) From Lemma \ref{fn}, we need to consider $Q$ such that $h(-4Q)\leq 2$.
However by Proposition \ref{number} one can check that there don't exist such a $Q$ that $h(-4Q)=1$ in these cases.
Therefore the result follows.
\end{proof}

\section{ Computations for the exceptional cases}\label{comp}

In this section, we  completely determine when the fixed points of the full Atkin-Lehner involution $W_N$ are Weierstrass points on $X_0(N)$ including the exceptional cases listed in Lemma~\ref{full}.
From now on, we always assume that $N$ is one of those integer values.
Let $\tau\in \mathfrak H$ so that $[\tau]:=\Gamma_0(N)\tau\in X_0(N)$ is a fixed point of $W_N$.
Let $\{f_1,f_2,\dots,f_g\}$ be a basis of the cuspform space $S_2(N)$ of weight 2 on $\Gamma_0(N)$, where $g$ is the genus of $X_0(N)$.
Then the differentials $\omega_i:=f_idz$ form a basis for the holomorphic differentials on $X_0(N)$.
Since ${\rm ord}_{[\tau]}(\omega_i)={\rm ord}_{[\tau]}(f_i)$, we will compute the Weierstrass weight of $[\tau]$ by using the $f_i$.
In general, it is very difficult to compute the explicit forms of the $f_i$, and hence we will use the Fourier expansions at the infinite cusp $\infty$ of the $f_i$ which can be easily computed by using the computer algebra system including \textsc{Sage}.

In general,  the orders ${\rm ord}_{[\tau]}(f_i)$ do not satisfy the following:
\begin{equation}\label{ord-ineq}
0={\rm ord}_{[\tau]}(f_1)<{\rm ord}_{[\tau]}(f_2)<\cdots<{\rm ord}_{[\tau]}(f_g).
\end{equation}
However, any basis satisfying \eqref{ord-ineq} can be expressed as a linear combination of the $f_i$, and hence we have the following result:
\begin{lem}\label{cuspform}
Let $[\tau]$ be a fixed point of $W_N$ on $X_0(N)$, and let $\{f_1,f_2,\dots,f_g\}$ be a basis of the cusp form space $S_2(N)$ of weight 2 on $\Gamma_0(N)$. Then $[\tau]$ is a Weierstrass point of $X_0(N)$ if and only if $\det(A)=0$, where
$A=\left [\frac{f_i^{(j-1)}(\tau)}{(j-1)!}\right ]$ is a $g\times g$ matrix with $1\leq i,j\leq g$.
\end{lem}
\begin{proof}
First we consider the Taylor expansions of $f_i(z)$ at $z=\tau$ for $i=1,\dots,g$, which are given by
\begin{align*}
f_1(z) &= f_1(\tau)+ f_1^{'}(\tau)(z-\tau)+\cdots
+\frac{f_1^{(g-1)}(\tau)}{(g-1)!}(z-\tau)^{g-1}+\cdots \\
f_2(z) &= f_2(\tau)+ f_2^{'}(\tau)(z-\tau)+\cdots+\frac{f_2^{(g-1)}(\tau)}{(g-1)!}(z-\tau)^{g-1}+\cdots \\
 & \hspace{2cm}  \vdots \\
f_g(z) &= f_g(\tau)+ f_g^{'}(\tau)(z-\tau)+\cdots
+\frac{f_g^{(g-1)}(\tau)}{(g-1)!}(z-\tau)^{g-1}+\cdots.
\end{align*}
By applying the elementary row operations to the above expansions we can find cusp forms $h_1, h_2, \dots, h_g$ which are linear combinations of $f_1,f_2,\dots,f_g$ with Taylor expansions of the form
\begin{eqnarray*}
h_1(z) &=& h_1(\tau) +  h_1^{'}(\tau)(z-\tau)+\cdots
+ \frac{f_1^{(g-1)}(\tau)}{(g-1)!}(z-\tau)^{g-1}+\cdots \\
h_2(z) &=& \hspace{1.4cm} h_2^{'}(\tau)(z-\tau)+\cdots  +  \frac{h_2^{(g-1)}(\tau)}{(g-1)!}(z-\tau)^{g-1}+\cdots \\
 & & \hspace{3.3cm}   \ddots  \\
h_g(z) &= & \hspace{4.9cm}
\frac{h_g^{(g-1)}(\tau)}{(g-1)!}(z-\tau)^{g-1}+\cdots
\end{eqnarray*}
that satisfy
$$0=\ord_P(\omega_1)<\ord_P(\omega_2)<\cdots <\ord_P(\omega_g),$$
where $P=[\tau]$ and $\omega_i=h_i dz$ for $i=1,\dots,g$.
Thus we see that $\det (A)=c\prod_{j=1}^g \frac{h_j^{(j-1)}(\tau)}{(j-1)!} $ for some nonzero constant $c$ and observe that
\begin{align*}
\det (A) =0 & \iff  \ord_P(\omega_j)\ge j \hbox{ for some } j \\
  & \iff  {\rm wt}(P)> 0 \\& \iff P
  \hbox{  is a Weierstrass point.}
\end{align*}
\end{proof}

Now we  estimate the value $\frac{f_i^{(j-1)}(\tau)}{(j-1)!}$ by using the Fourier expansion of $f_i$ at $\infty$ $f_i=\sum_{n=1}^\infty a_i(n)q^n$ where $z\in \mathfrak H$ and $q=e^{2\pi i z}$.
Note that
\begin{equation}\label{Fourier}
\frac{f_i^{(j-1)}(\tau)}{(j-1)!}=\sum_{n=1}^\infty\frac{(2\pi i n)^{j-1}}{(j-1)!}a_i(n)q^n\big|_{z=\tau}.
\end{equation}
Since $|e^{2\pi i \tau}|=e^{-2\pi {\rm Im}(\tau)}$, we need to find $\tau'\in[\tau]$ whose imaginary part is as big as possible for the fast convergence of the right side of \eqref{Fourier}.
We will consider an algorithm to find $\tau'\in[\tau]$ so that ${\rm Im}(\tau')$ is the biggest.
Note that ${\rm Im}(\alpha(\tau))=\frac{{\rm Im}(\tau)}{|c\tau+d|}$ for $\alpha=\sm a&b\\c&d\esm\in\Gamma_0(N)$.
Thus it suffices to find a pair $(c,d)$ with $c\equiv 0\pmod N$ and $\gcd(c,d)=1$ so that $|c\tau+d|$ is the smallest.
Let $\tau=x+iy$ with $x,y\in\R$.
Since $|c\tau+d|^2=(cx+d)^2+(cy)^2\leq 1$, there are only finitely many such pairs $(c,d)$ because $|c|\leq\frac{1}{|y|}$ and $|cx+d|\leq 1$.

\

We summarize the procedures explained in the above as the following algorithm.

\begin{alg}\label{alg2}
The following steps implement an algorithm  to find $\tau'\in[\tau]$ so that ${\rm Im}(\tau')$ is the biggest.
\par\noindent
{\bf STEP 1.} Set $m=[1/|y|]$ the largest integer not greater than $1/|y|$.
\par\noindent
{\bf STEP 2.} For each $c$ with $-m\leq c\leq m$ and $c\equiv0 \pmod N$, set $l_c=[-1-cx]$ and $u_c=[1-cx]$.
\par\noindent
{\bf STEP 3.} Pick a pair $(c_0,d_0)$ so that $|c_0\tau+d_0|$ is the smallest among $|c\tau+d|$ where $-m\leq c\leq m$ and $l_c\leq d\leq u_c$ with $\gcd(c,d)=1$.
\par\noindent
{\bf STEP 4.} Set $\tau'=\frac{{\rm Im}(\tau)}{|c_0\tau+d_0|}$.
\end{alg}

We list the fixed points of $W_N$ in Table \ref{table} by using Algorithm~\ref{alg1} and Algorithm~\ref{alg2}.
By using Lemma~\ref{cuspform}, we can confirm that there are no Weierstrass points on $X_0(N)$ arising from the fixed points of $W_N$. We have used \textsc{Sage}, \textsc{Maple}, and \textsc{Mathematica} for the numerical computations.


\begin{center}
\begin{longtable}{c|l}
\caption{Fixed points of $W_N$ on $X_0(N)$}\label{table}\\
$N$  & fixed points of $W_N$ \\
\hline \\
$22$ & $\frac{1}{\sqrt{-22}}, -\frac{6}{13}+\frac{\sqrt{-22}}{286}$\\
$28$ & $\frac{1}{2\sqrt{-7}}, -\frac{8}{11}+\frac{\sqrt{-7}}{154}$\\
$30$ & $\frac{1}{\sqrt{-30}}, \frac{8}{17}+\frac{\sqrt{-30}}{510}, \frac{4}{13}+\frac{\sqrt{-30}}{390}, \frac{2}{11}+\frac{\sqrt{-30}}{330}$\\
$33$ & $\frac{1}{\sqrt{-33}}, \frac{1}{2}+\frac{\sqrt{-33}}{66}, \frac{9}{14}+\frac{\sqrt{-33}}{462}, -\frac{2}{7}+\frac{\sqrt{-33}}{231}$\\
$34$ & $\frac{1}{\sqrt{-34}}, \frac{9}{19}+\frac{\sqrt{-34}}{646}, -\frac{4}{5}+\frac{\sqrt{-34}}{170}, \frac{4}{5}+\frac{\sqrt{-34}}{170}$\\
$37$ & $\frac{1}{\sqrt{-37}}, \frac{1}{2}+\frac{\sqrt{-37}}{74}$\\
$40$ & $\frac{1}{2\sqrt{-10}}, -\frac{6}{11}+\frac{\sqrt{-10}}{220}, -\frac{8}{13}+\frac{\sqrt{-10}}{260}, \frac{5}{7}+\frac{\sqrt{-10}}{140}$\\
$42$ & $\frac{1}{\sqrt{-42}}, \frac{11}{23}+\frac{\sqrt{-42}}{966}, \frac{11}{17}+\frac{\sqrt{-42}}{714}, -\frac{2}{13}+\frac{\sqrt{-42}}{546}$\\
$43$ & $\frac{1}{\sqrt{-43}}, -\frac{3}{4}+\frac{\sqrt{-43}}{172}, \frac{3}{4}+\frac{\sqrt{-43}}{172}, \frac{1}{2}+\frac{\sqrt{-43}}{86}$\\
$45$ & $\frac{1}{3\sqrt{-5}}, \frac{1}{2}+\frac{\sqrt{-5}}{30}, \frac{11}{14}+\frac{\sqrt{-5}}{210}, \frac{4}{7}+\frac{\sqrt{-5}}{105}$\\
$46$ & $\frac{1}{\sqrt{-46}}, \frac{12}{25}+\frac{\sqrt{-46}}{1150}, \frac{3}{5}+\frac{\sqrt{-46}}{230}, -\frac{3}{5}+\frac{\sqrt{-46}}{230}$\\
$48$ & $\frac{1}{4\sqrt{-3}}, -\frac{13}{19}+\frac{\sqrt{-3}}{228}, -\frac{7}{13}+\frac{\sqrt{-3}}{156}, -\frac{1}{7}+\frac{\sqrt{-3}}{84}$\\
$52$ & $\frac{1}{2\sqrt{-13}}, -\frac{13}{17}+\frac{\sqrt{-13}}{442}, \frac{4}{7}+\frac{\sqrt{-13}}{182}, -\frac{4}{7}+\frac{\sqrt{-13}}{182}$\\
$57$ & $\frac{1}{\sqrt{-57}}, \frac{1}{2}+\frac{\sqrt{-57}}{114}, -\frac{15}{22}+\frac{\sqrt{-57}}{1254}, -\frac{4}{11}+\frac{\sqrt{-57}}{627}$\\
$58$ & $\frac{1}{\sqrt{-58}}, \frac{15}{31}+\frac{\sqrt{-58}}{1798}$\\
$60$ & $\frac{1}{2\sqrt{-15}}, \frac{15}{23}+\frac{\sqrt{-15}}{690}, \frac{14}{19}+\frac{\sqrt{-15}}{570}, \frac{10}{17}+\frac{\sqrt{-15}}{510}$\\
$64$ & $\frac{1}{8\sqrt{-1}}, -\frac{9}{17}+\frac{\sqrt{-1}}{136}, -\frac{4}{5}+\frac{\sqrt{-1}}{40}, \frac{4}{5}+\frac{\sqrt{-1}}{40}$\\
$67$ & $\frac{1}{\sqrt{-67}}, -\frac{3}{4}+\frac{\sqrt{-67}}{268}, \frac{3}{4}+\frac{\sqrt{-67}}{268}, \frac{1}{2}+\frac{\sqrt{-67}}{134}$\\
$70$ & $\frac{1}{\sqrt{-70}}, \frac{18}{37}+\frac{\sqrt{-70}}{2590}, \frac{15}{19}+\frac{\sqrt{-70}}{1330}, \frac{12}{17}+\frac{\sqrt{-70}}{1190}$\\
$72$ & $\frac{1}{6\sqrt{-2}}, \frac{9}{19}+\frac{\sqrt{-2}}{228}, -\frac{2}{17}+\frac{\sqrt{-2}}{204}, \frac{3}{11}+\frac{\sqrt{-2}}{132}$\\
$73$ & $\frac{1}{\sqrt{-73}}, \frac{1}{2}+\frac{\sqrt{-73}}{146}, \frac{4}{7}+\frac{\sqrt{-73}}{511}, -\frac{4}{7}+\frac{\sqrt{-73}}{511}$\\
$78$ & $\frac{1}{\sqrt{-78}}, \frac{20}{41}+\frac{\sqrt{-78}}{3198}, \frac{19}{29}+\frac{\sqrt{-78}}{2262}, -\frac{16}{19}+\frac{\sqrt{-78}}{1482}$\\
$82$ & $\frac{1}{\sqrt{-82}}, \frac{21}{43}+\frac{\sqrt{-82}}{3526}, \frac{5}{7}+\frac{\sqrt{-82}}{574}, -\frac{5}{7}+\frac{\sqrt{-82}}{574}$\\
$85$ & $\frac{1}{\sqrt{-85}}, \frac{1}{2}+\frac{\sqrt{-85}}{170}, \frac{13}{22}+\frac{\sqrt{-85}}{1870}, -\frac{2}{11}+\frac{\sqrt{-85}}{935}$\\
$88$ & $\frac{1}{2\sqrt{-22}}, -\frac{12}{23}+\frac{\sqrt{-22}}{1012}, -\frac{12}{19}+\frac{\sqrt{-22}}{836}, -\frac{10}{13}+\frac{\sqrt{-22}}{572}$\\
$93$ & $\frac{1}{\sqrt{-93}}, \frac{1}{2}+\frac{\sqrt{-93}}{186}, -\frac{23}{34}+\frac{\sqrt{-93}}{3162}, \frac{11}{17}+\frac{\sqrt{-93}}{1581}$\\
$97$ & $\frac{1}{\sqrt{-97}}, \frac{1}{2}+\frac{\sqrt{-97}}{194}, -\frac{6}{7}+\frac{\sqrt{-97}}{679}, \frac{6}{7}+\frac{\sqrt{-97}}{679}$\\
$100$ & $\frac{1}{10\sqrt{-1}}, -\frac{22}{29}+\frac{\sqrt{-1}}{290}, \frac{6}{13}+\frac{\sqrt{-1}}{130}, \frac{7}{13}+\frac{\sqrt{-1}}{130}$\\
$102$ & $\frac{1}{\sqrt{-102}}, \frac{26}{53}+\frac{\sqrt{-102}}{5406}, -\frac{25}{37}+\frac{\sqrt{-102}}{3774}, \frac{19}{23}+\frac{\sqrt{-102}}{2346}$\\
$112$ & $\frac{1}{4\sqrt{-7}}, -\frac{15}{29}+\frac{\sqrt{-7}}{812}, \frac{13}{23}+\frac{\sqrt{-7}}{644}, -\frac{7}{11}+\frac{\sqrt{-7}}{308}$\\
$130$ & $\frac{1}{\sqrt{-130}}, \frac{33}{67}+\frac{\sqrt{-130}}{8710}, -\frac{25}{31}+\frac{\sqrt{-130}}{4030}, \frac{16}{23}+\frac{\sqrt{-130}}{2990}$\\
$133$ & $\frac{1}{\sqrt{-133}}, \frac{1}{2}+\frac{\sqrt{-133}}{266}, -\frac{15}{26}+\frac{\sqrt{-133}}{3458}, \frac{11}{13}+\frac{\sqrt{-133}}{1729}$\\
$142$ & $\frac{1}{\sqrt{-142}}, \frac{36}{73}+\frac{\sqrt{-142}}{10366}, -\frac{10}{11}+\frac{\sqrt{-142}}{1562}, \frac{10}{11}+\frac{\sqrt{-142}}{1562}$\\
$148$ & $\frac{1}{2\sqrt{-37}}, -\frac{31}{41}+\frac{\sqrt{-37}}{3034}, -\frac{10}{19}+\frac{\sqrt{-37}}{1406}, \frac{10}{19}+\frac{\sqrt{-37}}{1406}$\\
$163$ & $\frac{1}{\sqrt{-163}}, -\frac{3}{4}+\frac{\sqrt{-163}}{652}, \frac{3}{4}+\frac{\sqrt{-163}}{652}, \frac{1}{2}+\frac{\sqrt{-163}}{326}$\\
$177$ & $\frac{1}{\sqrt{-177}}, \frac{1}{2}+\frac{\sqrt{-177}}{354}, \frac{41}{62}+\frac{\sqrt{-177}}{10974}, \frac{10}{31}+\frac{\sqrt{-177}}{5487}$\\
$190$ & $\frac{1}{\sqrt{-190}}, \frac{48}{97}+\frac{\sqrt{-190}}{18430}, -\frac{26}{43}+\frac{\sqrt{-190}}{8170}, \frac{26}{29}+\frac{\sqrt{-190}}{5510}$\\
$193$ & $\frac{1}{\sqrt{-193}}, \frac{1}{2}+\frac{\sqrt{-193}}{386}, -\frac{8}{11}+\frac{\sqrt{-193}}{2123}, \frac{8}{11}+\frac{\sqrt{-193}}{2123}$\\
$232$ & $\frac{1}{2\sqrt{-58}}, -\frac{30}{59}+\frac{\sqrt{-58}}{6844}, \frac{23}{37}+\frac{\sqrt{-58}}{4292}, \frac{23}{31}+\frac{\sqrt{-58}}{3596}$\\
$253$ & $\frac{1}{\sqrt{-253}}, \frac{1}{2}+\frac{\sqrt{-253}}{506}, -\frac{31}{34}+\frac{\sqrt{-253}}{8602}, -\frac{14}{17}+\frac{\sqrt{-253}}{4301}$
\end{longtable}
\end{center}

{\bf Acknowledgments}. We would like to thank KIAS (Korea Institute for Advanced Study) for its hospitality while we have worked on this result.
Daeyeol Jeon would like to thank Brown University for its hospitality during his sabbatical year.

\end{document}